\documentclass[12pt, a4paper]{amsart}
\usepackage{amssymb, amsthm, amsmath, amsfonts}
\usepackage[marginpar=1cm, margin=2.6cm]{geometry}
\usepackage{verbatim} 
\usepackage{xcolor}

\theoremstyle{plain}
\newtheorem{thm}{Theorem}
\newtheorem{lem}{Lemma}
\newtheorem{cor}{Corollary}
\newtheorem{prop}{Proposition}

\theoremstyle{definition}
\newtheorem*{definitions}{Definitions}

\theoremstyle{remark}
\newtheorem*{remark}{Remark}
\newtheorem*{remarks}{Remarks}

\newcommand{\x}{\times}
\newcommand{\al}{\alpha}
\newcommand{\del}{\partial}
\renewcommand{\th}{\theta}
\newcommand{\Si}{\Sigma}
\newcommand{\la}{\lambda}
\newcommand{\be}{\beta}
\newcommand{\ep}{\epsilon}

\newcommand{\s}{\mathbf s}

\newcommand{\R}{\mathbb{R}}

\newcommand{\Z}{\mathbb{Z}}

\newcommand{\cs}{{\# }}

\newcommand{\Ff}{{\mathcal F}}

\newcommand{\Mm}{{\mathcal  M}}

\newcommand{\PP}{\mathbb{P}}

\newcommand{\ZZ}{\mathbb{Z}}

\newcommand{\ee}{{\bf e}}

\newcommand{\ww}{{\bf w}}

\newcommand{\HG}{{\mathfrak H}}

\newcommand*\sm[1]{\left(\begin{smallmatrix}#1\end{smallmatrix}\right)}

\title{Framing $3$-manifolds with bare hands}

\begin{document}
\author{Riccardo Benedetti}
\author{Paolo Lisca}
\address{Dipartimento di Matematica, Largo Bruno Pontecorvo 5, 56127 Pisa, Italy} 
\email{riccardo.benedetti@unipi.it,\  paolo.lisca@unipi.it}

\subjclass[2010]{57R25 (57M99)}

\date{\today}

\begin{abstract}
After surveying existing proofs that every closed, orientable $3$-manifold is parallelizable, 
we give three proofs using minimal background. In particular, our proofs use neither spin structures 
nor the theory of Stiefel-Whitney classes.
\end{abstract}

\maketitle

\section{Introduction}\label{s:intro}

The aim of this note is to provide three proofs ``with bare hands'' of the
following primary result in $3$-dimensional differential topology,
originally attributed to Stiefel \cite{Sti} (1936):

\begin{thm}\label{t:main}
Every orientable,  closed $3$-manifold is parallelizable.
\end{thm}

We realized by searching the literature that there are  at least four 
modern proofs of the above result, collected  in~\cite{Ge, FM} in a very clean way. 
Each of those proofs requires a somewhat robust mathematical background, so 
we asked ourselves whether there might be a proof which uses minimal background.\footnote
{Essentially the same question was asked in the Mathematics StackExchange Forum, see\\ 
{\tt https://math.stackexchange.com/questions/1107682/elementary-proof-of-the-fact-that-\\
any-orientable-3-manifold-is-parallelizable}, but the answers given there until July 18, 2018
use the same tools employed in the  proofs mentioned above.} 
By asking the use of `minimal background' we meant that such 
a proof should (i) satisfy the qualitative constraint of adopting a minimal toolbox (the simplest properties of 
cohomology and homotopy groups, the basic tools of differential topology and transversality theory 
such as given  e.g.~in~\cite{M} or~\cite{GP} and a few well-known facts about vector bundles and their 
Euler classes) and (ii) be as self-contained as possible. Eventually we found three such proofs 
which, contrary to some of the proofs present nowadays in the literature, 
make neither explicit reference to Stiefel-Whitney classes nor to spin structures.

Throughout the paper, $M$ denotes an orientable, closed (i.e.~compact without boundary) 
smooth $3$-manifold. It is not restrictive to assume that $M$ is  connected as well.  
Recall that a {\em  combing} of $M$ is a nowhere vanishing tangent vector
field on $M$. Moreover, $M$ is {\em parallelizable} if it admits
a {\em framing}, that is a triple $\Ff=(w,z,v)$ of pointwise linearly independent combings. 
The existence of a framing is equivalent to the existence of a {\em trivialization}
\[
\tau_\Ff: M\times \R^3 \to TM
\]
of the tangent bundle of $M$. A
framing incorporates an {\em orientation} of $M$ and, {\em vice versa}, 
if $M$ is oriented and parallelizable, then there are framings inducing the given orientation.  
We will always assume that $M$ is {\it oriented}, with a fixed auxiliary orientation.

The paper is organized as follows. In Section~\ref{s:3proofs} we briefly recall the four proofs collected in~\cite{Ge, FM} and  
we point out why they do not satisfy our minimal background requirements. 
In Section~\ref{s:results} we fix some notation and we recall a few well-known bare hands results. 
In each one of Sections~\ref{s:barehands}, \ref{s:barehands2} and~\ref{s:barehands3}  we give 
a different bare hands proof of Theorem~\ref{t:main}. The proof of Section~\ref{s:barehands}  
is purely $3$-dimensional and could be regarded as a minimalistic version of the available modern 
proof based on Stiefel-Whitney classes. 
The proofs provided in Sections~\ref{s:barehands2} and~\ref{s:barehands3} could also be regarded as 
minimalistic versions of available modern proofs based on even surgery presentations and, respectively, 
$4$-dimensional and purely $3$-dimensional considerations. In particular, the proof of Section~\ref{s:barehands3} could 
be viewed as a simplification of the available modern proof mainly based on spin structures.

\noindent
{\bf Acknowledgements.} The authors heartfully thank Chuck Livingston for pointing out, after the appearance of 
a previous version of this note, the existence of a proof of Theorem~\ref{t:main} based on Kaplan's 
results~\cite{Ka}. We are also grateful to Alexis Marin for an interesting e-mail exchange on the topics at hand.

\section{Available modern proofs of Theorem~\ref{t:main}}\label{s:3proofs}

Each of the four modern proofs we survey in this section  
argues that $M$ admits a {\em quasi-framing}, 
that is a framing $\Ff_0$ of a submanifold $M_0$ of $M$ of the form
\[
M_0 = M \setminus {\rm Int}(B),
\]
where $B$ is a smooth $3$-disk
embedded in $M$. The quasi-framing $\Ff_0$ can be extended with bare hands to a framing of
the whole of $M$ as follows. By the uniqueness of disks up to ambient
isotopy, the choice of $B$ is immaterial. Hence, we can
assume that $B$ is contained in a chart of $M$ and looks standard therein. 
Upon fixing an auxiliary metric on $M$ and a trivialization of $TM$ over $B$, 
the restriction of $\Ff_0$ to $S^2=\partial B$ is encoded by a smooth map
\[
\rho: S^2 \to SO(3).
\]
Since the universal covering space of $SO(3)\cong\PP^3(\R)$ is $S^3$, we have  
$\pi_2(SO(3))=\pi_2(S^3)=0$, therefore $\rho$ can be extended over $B$ and 
$\Ff_0$ to $M$.

\subsection{The three proofs presented in~\cite{Ge}}\label{ss:Ge-proofs}

We refer the reader to~\cite[\S~4.2]{Ge} for details. The first and third proofs presented in~\cite{Ge} 
use a certain mixture of the theory of Stiefel-Whitney classes and spin structures to establish 
the existence of a quasi-framing as follows. The first Stiefel-Whitney class $w_1(M)$ vanishes 
because $M$ is orientable, and the key point in both proofs is to show 
that $w_2(M)$ vanishes as well. 
Using obstruction theory to define Stiefel-Whitney classes one 
can argue that $w_2(M)=0$ implies the existence of a spin structure on $M$, 
and therefore that $M$ admits a quasi-framing. The first and third proofs differ in 
the way they establish the vanishing of $w_2(M)$.

The first proof, resting on several properties of Stiefel-Whitney classes, 
is perhaps the one requiring the most sophisticated background. 
The so-called {\em Wu classes} $v_i\in H^i(M;\Z/2\Z)$ can be characterized 
by the property that, for every $x\in H^{3-i}(M;\Z/2\Z)$, 
\[
\langle {\rm Sq}^i(x), [M]\rangle = \langle v_i\cup x, [M]\rangle,
\]
where $[M]$ denotes the fundamental class of $M$ in $H_3(M;\Z/2\Z)$ 
and ${\rm Sq}^i$ is the $i$-th {\em Steenrod square} operation. It follows that 
$v_0= 1$ and, for dimensional reasons, $v_i=0$ if $i> 3-i$. Hence,  
the only potentially nonzero Wu classes are $v_0$ and $v_1$. Moreover,  
Wu classes and Stiefel-Whitney classes are related through {\em Wu's formula}:
\[
w_q(M)= \sum_{i+j=q} {\rm Sq}^i (v_j).
\]
Since ${\rm Sq}^0$ is the identity map and ${\rm Sq}^i (x)=0$ when $i>\deg(x)$, 
by Wu's formula we have 
\[
0=w_1(M)={\rm Sq}^0 (v_1) + {\rm Sq}^1 (v_0) = v_1.
\]
By Wu's formula again, the vanishing of $v_1$ implies $w_2(M)=0$. 

The third proof given in~\cite[\S~4.2]{Ge} goes as follows: first one shows~\cite[Lemma~4.2.2]{Ge} 
that if $\Si$ is a closed, possibly non orientable surface embedded in $M$, 
then $w_2(E)=0$ where $E=E(\Si)$ is a tubular neighborhood of $\Si$ in $M$. 
The proof is elementary modulo the use of the basic {\it Whitney sum formula} 
for Stiefel-Whitney classes of vector bundles.  
The conclusion is entirely based on the theory of spin structures combined with 
some bare hands reasoning. It is a slight simplification of the proof proposed by R. Kirby in
\cite{Ki2}. The argument is by contradiction: if $w_2(M)\neq 0$ then its
Poincar\'e dual in $H_1(M;\Z/2\Z)$ is represented by a knot $K$
embedded in $M$. Then, the assumption implies that:
\begin{enumerate}
\item[(a)]
$M\setminus K$ carries a spin structure $\s$ which cannot be
 extended over any embedded $2$-disk transverse to $K$;
\item[(b)]
there is a compact, closed surface $\Si$ embedded in $M$ 
intersecting $K$ transversely in a single point $x_0$.
\end{enumerate}
By the general theory of spin structures, the vanishing of $w_2(E)$ implies that 
the set of spin structures on $E=E(\Si)$ is non-empty, and in fact  
it is an affine space on  
\[
H^1(E;\Z/2\Z) \cong H^1(\Si;\Z/2\Z)\cong H^1(\Si\setminus\{x_0\};\Z/2\Z).
\]
It follows that the restriction of $\s$ to $E\setminus K$ extends to the 
whole of $E$, contradicting (b).

The second proof presented in~\cite[\S~4.2]{Ge} is less standard. It is based 
on the following non-trivial fact due to Hilden, Montesinos and Thickstun~\cite{HMT}:
{\em there exists a branched covering map $\pi: M\to S^3$ such that the branching locus
bounds an embedded $2$-disk in $M$}. Using this fact, it is relatively easy 
to lift a framing of $S^3$, which can be constructed directly, to a quasi-framing of $M$. Although this proof is of a 
geometric-topological nature, clearly it does not use minimal background.

\begin{remark}
The three proofs described above are quite demanding from 
our ``bare hands'' point of view. The first proof, in particular, comes out of a relatively  
obscure algebraic machinery -- we would have a hard time deducing from such a machinery 
a heuristic justification for the existence of framings on closed 3-manifolds. 
\end{remark}

\subsection{The proof presented in~\cite{FM}}\label{ss:FM-proof}
We refer the reader to~\cite[\S 9]{FM} for details. The starting point is 
the {\it Lickorish-Wallace  theorem}~\cite{Li, Wa}, stating that the $3$-manifold $M$ can be obtained by 
surgery along a framed link $L\subset S^3$. Equivalently, the statement says that $M$ is the boundary 
of a $4$-manifold $W$ constructed by attaching $4$-dimensional $2$-handles to the $4$-ball.
Then, an argument essentially due to Kaplan~\cite{Ka} shows that by applying {\it Kirby moves} to $L$, 
it is not restrictive to assume that all the framings of $L$ are even. By using this fact one shows 
that the $4$-manifold $W$ is parallelizable, hence that $M=\partial W$ is {\it stably-parallelizable}
and eventually admits a quasi framing. 

\begin{remark}\label{discussion} The proof presented in~\cite{FM} satisfies to a large extent the first
minimal background requirement from Section~\ref{s:intro}. In fact: the final portion of the argument, 
which will be recalled in Section \ref {s:barehands2}, is ``bare hands"; Rourke's proof \cite{Ro} of the Lickorish-Wallace theorem
is completely elementary and constructive provided one allows the use of  {\it Smale's theorem}~\cite{Sm}   
so that, for example, one can take for granted that the operation of cutting and re-gluing a $3$-ball does not change a $3$-manifold; 
although Kaplan's argument requires the introduction of Kirby calculus, it does {\it not} use the hard part of Kirby's theorem~\cite{Ki1} 
on the completeness of the calculus. Everything  considered, we think that the proof 
presented in~\cite{FM} is not as self-contained as possible and therefore it does not satisfy the second 
minimal background requirement from Section~\ref{s:intro}.
\end{remark}

\section{Some notation and bare hands results}\label{s:results}

In this section we collect some notation and a few well-known facts that we 
allow in our minimal toolbox. Let $N$ be a closed, connected manifold of dimension $n$, and let 
\[
\xi: B\to N
\]
be a vector bundle of rank $k$, considered up to bundle isomorphisms. 
According to our bare hands constraints, in this generality the only allowable
``characteristic'' class of $\xi$ is 
\[
\ww(\xi) \in H^k(N;\Z/2\Z)\cong H_{n-k}(N;\Z/2\Z),
\]
defined as the class carried by the transverse self-intersection of $N$ viewed as
the {\em zero section} of $\xi$ inside $B$. The class $\ww(\xi)$ actually coincides with the $k$-th 
Stiefel-Whitney class $ w_k(\xi)$, but we shall not need this fact. Moreover, we 
will not make use of any other Stiefel-Whitney class. If both $N$ and $\xi$ are
oriented, the same construction defines an integral class   
\[
\ee(\xi)\in H^k(N;\Z),
\]
sent to $\ww(\xi)$ by the natural map $H^k(N;\Z)\to H^k(N;\Z/2\Z)$. 
In both cases we talk about the {\it Euler class} of $\xi$, referring to 
either $\ww(\xi)$ or $\ee(\xi)$ depending on the context. 

We will feel free to use the following facts:
\begin{itemize}
\item 
if $\xi=\xi_1 \oplus \xi_2$ is the Whitney sum of two vector bundles then
$\ww(\xi) = \ww(\xi_1)\cup \ww(\xi_2)$;
\item
a line bundle $\la$ on $N$ has a nowhere vanishing section if and only if $\ww(\la)=0$;
\item If $N$ is oriented, a rank-2 oriented vector bundle $\xi$ on $N$ has a nowhere vanishing
section if and only if $\ee(\xi)=0$;  
\item 
if $\xi=\la_1 \oplus \la_2$ is the Whitney sum of two line bundles then
\[
\ww(\det \xi)= \ww(\la_0)+\ww(\la_1);
\]
\item 
$\ww(TN)\in H^n(N;\Z/2\Z)$ and  
\[
\langle\ww(TN), [N]\rangle = \chi(N)\bmod (2) \in \Z/2\Z;
\]
\item
$N$ is orientable if and only if $\ww(\det TN)=0$;
\item 
If $N$ is oriented then
\[
\langle\ee(TN),[N]\rangle=\chi (N) \in \Z;
\]
\item 
let $M$ be a closed, oriented $3$-manifold and $\be\in H^2(M;\Z)$. Then, there is an oriented, connected, closed
$1$-submanifold $C\subset M$ which represents the Poincar\'e dual of $\be$. If $\be\in H^j(M;\Z/2\Z)$
with $0\leq j\leq 3$, there is a possibly non orientable, connected and  closed $(3-j)$-submanifold
of $M$ which represents the Poincar\'e dual of $\be$. Moreover, the cup product of two cohomology
classes $\be_1$ and $\be_2$ can be represented by a transverse intersection of submanifolds 
representing the Poincar\'e duals of $\be_1$ and $\be_2$;
\item 
any closed $3$-manifold $M$ carries a combing.~\footnote{This fact follows from $\chi(M)=0$ using the 
Poincar\'e-Hopf index theorem, clearly an allowable tool, together with the  fact that maps $S^2\to S^2$ 
are classified up to homotopy by their $\Z$-degree.}  
\end{itemize}

Given an auxiliary Riemannian metric $g$ on a closed $3$-manifold $M$, by normalization any 
combing of $M$ can be made of unitary norm, and by the Gram-Schmidt process any framing of $M$ can be turned into a point-wise $g$-orthonormal framing.
A unitary combing $v$ on $M$ determines an oriented distribution of tangent $2$-planes 
\[
F_v = \{F_v(x)\}_{x\in M}\subset TM, 
\]
where $F_v(x)\subset T_x M$ is the subspace $g(x)$-orthogonal to $v(x)$. We assume that, for each $x\in M$,  
$v(x)$ followed by an oriented basis of $F_v (x)$ gives an oriented basis of $T_x M$. 
The restriction of the projection $TM\to M$ gives rise to an oriented rank-2 real vector bundle 
$F_v \to M$ whose isomorphism type is independent of the choice of $g$ and depends on $v$
only up to homotopy. We denote by 
\[
\ee(F_v)\in H^2(M;\Z)  
\]
the Euler class of $F_v$.

\section{First bare hands proof of Theorem~\ref{t:main}}\label{s:barehands}

In this section we provide the first bare hands proof of Theorem~\ref{t:main}, 
resting neither on the theory of spin structures nor on properties of Stiefel-Whitney classes. 
Our tools consist of basic properties of cohomology groups, transversality theory, and 
the facts collected in Section~\ref{s:results}. We will also use the notation introduced in 
Section~\ref{s:results}.

The section is organized as follows. In Subsection~\ref{ss:comb-fram} we give a bare 
hands proof of the following proposition.
\begin{prop}\label{p:wFv=0} $M$ is parallelizable if and only if
there is a combing $v$ of $M$ such that $\ww(F_v)=0$, 
in which case $\ww(F_v)=0$  for every combing $v$.
\end{prop}

Proposition~\ref{p:wFv=0} reduces the proof of Theorem~\ref{t:main}  
to showing that $M$ carries a combing $v$ such that  $\ww(F_v)=0$. Observe that, 
for a combing $v$ on $M$, the property $\ww(F_v)=0$ is equivalent to the fact that, 
for every closed, connected, embedded surface $\Si\subset M$, we have 
\begin{equation}\label{e:w2FS=0}
\langle\ww(F_v), [\Si]\rangle = \langle \ww(F_v|_\Si), [\Si]\rangle = 0 \in \Z/2\Z.
\end{equation}
We claim that Equation~\eqref{e:w2FS=0} is a consequence of the equation 
\begin{equation}\label{e:WSbh}
\ww(F_v|_\Si) = \ww(T\Si) + \ww(\det T\Si)\cup \ww(\nu_\Si), 
\end{equation} 
where $\nu_\Si\subset TM$ denotes the normal line bundle of $\Si$.
In fact, if $\Si$ is orientable then $\ww(\det T\Si)=0$ and by~\eqref{e:WSbh} 
we have $\ww(F_v|_\Si)=\ww(T\Si)$. Therefore,
\[
\langle \ww(F_v|_\Si), [\Si]\rangle = 
\langle \ww(T\Si), [\Si]\rangle = \chi(\Si) \bmod (2) = 0\in\Z/2\Z.
\]
If $\Si$ is non-orientable then $\Si$ is homeomorphic to a connected sum $\#^h\PP^2(\R)$ of
$h$ copies of the projective plane.  Since $M$ is orientable, the normal line bundle $\nu_\Si$
is isomorphic to the determinant line bundle $\det T\Si$, and by~\eqref{e:WSbh}
we have 
\[
\langle \ww(F_v|_\Si), [\Si]\rangle = \chi(\Si)\bmod (2) + \langle \ww(\nu_\Si)\cup \ww(\nu_\Si), [\Si]\rangle 
= 2- h+h \bmod (2) = 0 \in\Z/2\Z. 
\]
Proposition~\ref{p:WSF} of Subsection~\ref{ss:WSF} below contains a bare hands proof of~\eqref{e:WSbh}, 
thus concluding our bare hands proof of Theorem~\ref{t:main}.

Before embarking in the bare hands proofs of Proposition~\ref{p:wFv=0} and~\ref{p:WSF} 
it seems worth to point out the existence of a short argument to prove $w_2(M)=0$ 
without spin structures, yielding a simplification of the first and third proofs from ~\cite[\S~4.2]{Ge}. 
More precisely, we prove Proposition~\ref{p:w2=0} below using only the existence of 
Stiefel-Whitney classes and the basic Whitney sum formula. 
\begin{prop}\label{p:w2=0}
Let $M$ be a closed, oriented 3-manifold. Then, $w_2(M) = 0$. 
\end{prop}

\begin{proof} 
Let $v$ be a combing on $M$ and $\Si\subset M$ a closed, connected, embedded surface. 
Then, we have the Whitney sum decompositions 
\[
TM|_\Si = F_v |_\Si \oplus \epsilon = T\Si \oplus \nu_\Si, 
\]
where $\epsilon$ is the trivial line bundle generated by $v$. 
 By the Whitney sum formula for Stiefel-Whitney classes,  the first decomposition
gives $\langle w_2(M), [\Si]\rangle = \langle w_2(F_v), [\Si]\rangle$, hence 
$w_2(M)=0$ if and only if $w_2(F_v)=0$. The second decomposition yields 
\begin{equation*}\label{e:WSF}
w_2(F_v|_\Si) = w_2(\Si) + w_1(\Si)\cup w_1(\nu_\Si),
\end{equation*} 
which is analogous to Equation~\eqref{e:WSbh}. An argument similar to the one above showing~\eqref{e:WSbh} 
$\Rightarrow$ \eqref{e:w2FS=0} gives $\langle w_2(F_v), [\Si]\rangle = 0\in\Z/2\Z$, therefore 
we conclude $w_2(F_v)=0$. 
\end{proof} 

\subsection{Combing and framing $3$-manifolds}\label{ss:comb-fram}
Our purpose in this subsection is to achieve a bare hands proof of Proposition~\ref{p:wFv=0} above. 

\begin{lem}\label{e=0} 
$M$ is parallelizable if and only if $\ee(F_v)=0$ for some unitary combing $v$.
\end{lem}

\begin{proof}
Let $v$ be a unitary combing of $M$ such that $\ee(F_v)=0$. 
Any nowhere vanishing section of $F_v$ can be normalized with respect to $g$ to a unitary 
section $w$ of $F_v$, extended to an oriented orthonormal framing $(w, z)$ of $F_v$ and 
finally to an oriented orthonormal framing $(w,z,v)$ of $M$. Conversely, for any 
orthonormal framing $(w,z,v)$ of $M$ we may view $v$ as a combing of $M$ and $w$ as 
a section of $F_v$.
\end{proof}

\subsection*{The comparison class} 

We can associate to an ordered pair of unitary combings $(v, v')$ of $M$ a 
smooth section $v\x v'$ of $F_v$ as follows. 
At a point $x\in M$ where $v(x) \neq\pm v'(x)$, $v\x v'(x) \in F_v(x)\subset T_x M$ is 
the ``vector product'' of $v(x)$ and $v'(x)$, i.e.~the only tangent vector such that
\begin{itemize}
\item 
$\|v\x v'(x)\|_{g(x)}^2 = 1 - g(v,v')^2$;
\item
$v\x v'(x)$ is $g(x)$-orthogonal to $v(x)$ and $v'(x)$; 
\item 
$(v(x), v'(x), v\x v'(x))$ is an oriented basis of $T_ x M$.
\end{itemize}
At a point $x\in M$ where $v(x) = \pm v'(x)$, we set $v\x v'(x) = 0$. 

If the two unitary combings $v$ and $v'$ are generic,
the section $v\x v'$ of $F_v$ is transverse to the zero section and
the zero locus 
\[
C := \{x\in M\ |\ v\x v'(x) = 0\} \subset M
\]
is a disjoint collection of simple closed curves. Moreover,  $C = C_+\cup C_-$, where 
\[
C_+ = \{x\in M\ |\ v(x) = v'(x)\}\quad\text{and}\quad
C_- = \{x\in M\ |\ v(x) = -v'(x)\}.
\]
By the very definition of $\ee(F_v)$, $C$ can be oriented to represent the Euler class of $F_v$.
Indeed, let $E(F_v)$ denote the total space of $F_v$, $M_0\subset E(F_v)$ the zero-section and 
$M_1 = v\x v'(M)\subset E(F_v)$. Under the natural identification of $M$ with $M_0$ the submanifold $C$ is identified 
with $M_0\cap M_1$. By transversality, for each $x\in M_0\cap M_1$ the natural projection 
$p_x : T_x E(F_v)\to F_v(x)$ maps isomorphically the image under $(v\x v)'_*$ of the fiber $N_x(C)$ 
of the normal bundle of $TC\subset TM|_C$ onto $F_v(x)$. Therefore, the given orientation on $F_v(x)$ 
can be pulled-back to $N_x(C)$ and, together with the orientation of $T_x M$, it induces an orientation on 
$T_x C$ in a standard way.

\begin{definitions}
An ordered pair of unitary combings $(v,v')$ of $M$ such that $v\x v'$ is a section of $F_v$ 
transverse to the zero section will be called a {\em generic pair} of unitary combings. 
We define the {\em comparison class} $\al(v,v')\in H^2(M;\ZZ)$ of a generic pair of unitary combings 
as the Poincar\'e dual of the homology class $[C_-]$ carried by the collection of curves $C_-$ 
with the orientation induced a a subset of the zero locus $C$ 
 of $v\x v' : M\to F_v$, where $C$ is oriented as described above, so that the resulting homology class
 $[C]$ represents the Poincar\'e dual of  $\ee(F_v)$.  
\end{definitions}

\theoremstyle{plain}
%\begin{remark}
%Standard transversality arguments show that the above definition is well--posed, i.e.~the 
%class $\al(v,v')$ is independent of the choice of generic representatives of $v$ and $v'$. 
%\end{remark}

\begin{lem}\label{l:comparison-properties}
Let $(v,v')$ be a generic pair of unitary combings of $M$. Then, 
\[
\al(v,v') = -\al(v',v)\quad\text{and}\quad \al(v,-v') = \al(v', -v).
\]
\end{lem} 

\begin{proof}
For each $x\in C$ the equality $F_v(x)=F_{v'}(x)$ holds, with the orientations of $F_v(x)$ and $F_{v'}(x)$ 
being the same or different according to, respectively, whether $x\in C_+$ or $x\in C_-$. We may choose 
a tubular neighborhood $U = U(C)$ such that the restrictions of the tangent plane fields $F_v |_U$ and $F_{v'} |_U$ 
are so close that there is a vector bundle isomorphism $\varphi : F_v |_U\stackrel{\cong}{\to} F_{v'} |_U$ 
which is the identity map on the intersections $F_v (x)\cap F_{v'} (x)$, $x\in U$, is orientation-preserving near 
$C_+ = \{x\in M\ |\ v(x) = v'(x)\}$ and orientation-reversing near 
$C_- = \{x\in M\ |\ v(x) = -v'(x)\}$. Since $\varphi\circ (v\x v') = v\x v' = - v'\x v$ and $- v'\x v$ is obtained 
by composing the section $v'\x v$ with the orientation-preserving automorphism of $F_{v'}$ given by minus 
the identity on each fiber, the orientation on $C_-$ as part of the zero locus of $v\x v' : M\to F_v$ is the opposite of its 
orientation as part of the zero locus of $v'\x v = -v\x v': M\to F_{v'}$.  This implies $\al(v,v') = -\al(v',v)$. 
Similarly, the orientation on $C_+$ as part of the zero locus of $v\x (-v') : M\to F_v$ coincides with its 
orientation as part of the zero locus of $(-v')\x v = v'\x (-v): M\to F_{v'}$, which implies $\al(v,-v') = \al(v',-v)$. 
\end{proof} 

\begin{lem}\label{l:compare} 
Let $(v,v')$ be a generic pair of unitary combings of $M$. Then, 
\[
\ee(F_v)-\ee(F_{v'})= 2\alpha(v,v').
\]
\end{lem}

\begin{proof}
According to the definitions we have 
\[
\ee(F_v) = \al(v,v') + \al(v,-v')\quad\text{and}\quad
\ee(F_{v'}) = \al(v',v) + \al(v',-v).
\]
The statement follows applying Lemma~\ref{l:comparison-properties} after taking the difference of the two equations.
\end{proof} 

\subsection*{Pontryagin surgery} 
Let $v$ be a unitary combing of $M$ and $C\subset M$ an oriented, simple closed curve such that 
the positive, unit tangent field along $C$ is equal to $v|_C$ and there is a trivialization
\[ 
j:D^2\times S^1 \stackrel{\cong}{\to} U(C)
\]
of a tubular neighborhood of $C$ in $M$ such that
\[
v \circ j= j_*(\partial /\partial \phi),
\]
where $\phi$ is a periodic coordinate on the $S^1$-factor 
of $D^2\x S^1$. Let  $(\rho, \th)$ be polar coordinates on the $D^2$-factor. Following 
terminology from~\cite{BP}, we say that a unitary combing $v'$ is obtained from $v$ by 
{\it Pontryagin surgery} along $C$ if, up to homotopy, $v'$ coincides with $v$ on $M\setminus U(C)$ and 
\[  
v' \circ j = j_*\left (-\cos(\pi\rho)\frac{\partial}{\del\phi} -\sin(\pi\rho)\frac{\partial}{\del\rho}\right)
\] 
on $U(C)$. 

\begin{remark}
A basic fact not used in this paper is that any two combings of $M$ are obtained 
from each other, up to homotopy, by Pontryagin surgery~\cite{BP}.
\end{remark}

\begin{lem}\label{pont-comb-surg} Let $v$ be a unitary combing of $M$ and
$\beta \in \ H^2(M;\ZZ)$. Then, possibly after a homotopy of $v$, there is a unitary 
combing $v'$ such that $(v,v')$ is a generic pair of unitary combings and 
\[
\alpha(v, v')= \beta. 
\]
\end{lem}

\begin{proof}   
Let $C\subset M$ be an oriented simple closed curve representing the Poincar\'e dual of $\beta$ 
and let $j : D^2\x S^1 \to U(C)$ be a trivialization of a neighborhood of $C$. 
Without loss of generality we may assume that  
the pull-back $j^*(g)$ of the auxiliary metric $g$ on $M$ is the standard product metric on $D^2\x S^1$.
After a suitable homotopy of $v$ the assumptions to perform Pontryagin surgery on $v$ 
along $C$ are satisfied. 
Consider a normal disc $D_{\phi_0} = j(D^2\x\{\phi_0\})$ and let $p = D_{\phi_0}\cap C$. Then, 
$T_p D_{\phi_0}$ coincides, as an oriented $2$-plane,  with $F_v(p)$  as well as with the $g(p)$-orthogonal subspace 
of $T_p C$ inside $T_p M$. Let $v'$ be a unitary combing obtained from 
$v$ by first performing a Pontryagin surgery on $U(C)$ and 
then applying a small generic perturbation supported on a small neighborhood of $M\setminus U(C)$. 
Then, $(v,v')$ is a generic pair of unitary combings and $C = \{x\in M\ |\ v(x) = -v'(x)\}$. 
By the definition of $\al(v,v')$, to prove the statement it suffices 
to show that the given orientation of $C$ coincides with its orientation as part of the zero set of $v\x v' : M\to F_v$.  
Near $C$ we have 
\[
(v\x v')\circ j = j_*\left(-\sin(\pi\rho) \frac{\del}{\del\th}\right) = 
j_*\left(\frac{\sin(\pi\rho)}{\rho}\left(y\frac{\del}{\del x} - x\frac{\del}{\del y}\right)\right),
\]
where $x = \rho\cos\th$ and $y=\rho\sin\th$ are rectangular coordinates on the $D^2$-factor. Observe that
$j_*$ sends the pair $(\del/\del x, \del/\del y)$ to an oriented framing of $F_v$. 
Using the resulting trivialization of $F_v$ we can write locally the restriction of $v\x v'$ to 
to the disc $D_{\phi_0}$ followed by projection onto $F_v$ as follows:
\[
v\x v'|_{D_{\phi_0}} : (x,y)\mapsto  \frac{\sin(\pi\rho)}{\rho} (y, -x) = \pi (y, -x)\ +\ \text{higher order terms}.
\]
It is easy to compute that $(v\x v')_* \circ j_*$ sends $\del/\del x$ to $-\pi\del/\del y$ and 
$\del/\del y$ to $\pi\del/\del x$, and since the matrix 
$\left(\begin{smallmatrix}0&\pi\\-\pi&0\end{smallmatrix}\right)$ has 
determinant $\pi^2>0$ this shows that the restriction of $(v\x v')_*$ to the normal bundle to $C$ 
composed with the projection onto $F_v$ is orientation-preserving along $C$, concluding the proof. 
\end{proof} 

We shall say that the Euler class $\ee(F_v)$ is {\em even} if there exists $\beta \in H^2(M;\Z)$
such that $\ee(F_v)=2\beta$. 

\begin{lem}\label{l:even} $M$ is parallelizable if and only 
$\ee(F_v)$ is even for every unitary combing $v$.
\end{lem}

\begin{proof}
If $M$ is parallelizable, then $M$ has a unitary framing $(w,z,v)$ and the class 
$\ee(F_v)=0$ is obviously even. Let $v'$ be an arbitrary unitary combing of $M$.
After possibly small perturbations of $v'$ and $v$ which do not change 
$\ee(F_{v'})$ nor $\ee(F_v)$, the pair $(v',v)$ becomes a generic 
pair of unitary combings and by Lemma~\ref{l:compare} we have 
\[
\ee(F_{v'}) =\ee(F_{v'}) - \ee(F_v) = 2\alpha(v',v).
\]
Therefore, $e(F_{v'})$ is even as well. 
Conversely, suppose that $v$ is a unitary framing with $\ee(F_v)=2\beta\in H^2(M;\Z)$. 
By Lemma~\ref{pont-comb-surg}, possibly after a homotopy of $v$ -- 
which does not change $\ee(F_v)$ -- 
there is a unitary framing $v'$ such that $(v,v')$ is a generic pair and 
$\alpha(v,v')=\beta$. Hence, by Lemma~\ref{l:compare} we have 
\[
\ee(F_v) - \ee(F_{v'})= 2\alpha(v,v')=2\beta,
\]
which implies $\ee(F_{v'})=0$, therefore $M$ is parallelizable by Lemma~\ref{e=0}.
\end{proof}

\begin{lem}\label{l:even2} Let $v$ be a unitary combing of $M$. Then, 
$\ee(F_v)$ is even if and only if $\ww(F_v)=0$.
\end{lem}

\begin{proof} 
The implication $\ee(F_v)=0$ $\Rightarrow$ $\ww(F_v)=0$ is trivial. We give two arguments for the other implication. 
The first argument uses a little bit of homological algebra. The short exact sequence of coefficients
$$0\to \Z \stackrel{2\cdot}{\to} \Z \to \Z/2\Z \to 0$$
induces a long exact sequence in cohomology including the segment 
\[
\cdots \to  H^2(M;\Z)\stackrel{2\cdot}{\to} H^2(M;\Z)\stackrel{\varphi}{\to} H^2(M;\Z/2\Z) \to \cdots
\]
where the map $\varphi$ is reduction mod $2$. Exactness yields the statement.

The second argument is more geometric. The Poincar\'e dual of $\ee(F_v)$ can be represented
by an oriented  knot $K\subset M$. If $\ww(F_v)=0$ then $K$ bounds an embedded surface $\Si\subset M$. 
If $\Si$ is orientable then $[K]=0$, hence $\ee(F_v)=0$, which is obviously even. If $\Si$ is non-orientable
then there is a collection $C$ of simple closed curves in the interior of $\Si$ such that 
$\Si\smallsetminus C$ is orientable and a tubular neighborhood $U$ of $C$ in $\Si$ is a union 
of M\"obius bands. Orient $\Si\smallsetminus{\stackrel{\circ}{U}}$ so that $K$ is an oriented boundary component  
and give $\del U$ the resulting boundary orientation. Orient the cores of $U$ so that the natural projection 
$\partial U \to C$ has positive degree. Then, $[K]=[\partial U]= 2[C]$, therefore $\ee(F_v)$ is even. 
\end{proof}

\begin{proof}[Proof of Proposition~\ref {p:wFv=0}]
The statement is an immediate consequence of Lemmas~\ref{l:even} and~\ref{l:even2}. 
\end{proof} 

\subsection{Proof of Equation~\eqref{e:WSbh}}\label{ss:WSF}
The purpose of this subsection is to give a bare hands proof of Proposition~\ref{p:WSF} below, 
which establishes Equation~\eqref{e:WSbh}. As explained at the beginning of the present section, 
this concludes our bare hands proof of Theorem~\ref{t:main}.

Let $v$ be a unitary combing of $M$ and $\Si\subset M$ a closed, embedded surface. 
At each point $x\in\Si$ we have the splittings 
\begin{equation}\label{e:splittings}
T_x M = F_v (x)\oplus \ep(x) = T_x \Si\oplus \nu_\Si (x),
\end{equation} 
where $\ep(x)$ is the (oriented) line spanned by $v(x)$, while $\nu_\Si (x)$
is the (unoriented) line orthogonal to $T_x \Si$.

\begin{prop}\label{p:WSF}
Let $v$ be a unitary combing of $M$ and $\Si\subset M$ a closed, embedded surface. Then, 
\begin{equation}\label{e:bhWSF}
\ww(F_v|_\Si) = \ww(T\Si) + \ww(\det T\Si)\cup \ww(\nu_\Si). 
\end{equation}
\end{prop}

\begin{proof}
Let $s:\Si\to F_v|_\Si$ be a generic section of the restriction $F_v$ to $\Si$. 
For each $x\in\Si$, the second splitting from~\eqref{e:splittings} induces decompositions
\[
s(x) = s_\Si(x)+ s_\nu(x), \quad v(x)= v_\Si(x) + v_\nu(x). 
\]
By transversality we may assume that:
\begin{enumerate}
\item[(i)]
the zero set $\{s=0\}\subset\Si$ consists of a finite number of points representing $\ww(F_v|_\Si)$;
\item[(ii)]
$s_\nu$ and $v_\nu$ are generic sections of $\nu_\Si$, so that both their zero sets 
$\{s_\nu = 0\}$ and $\{v_\nu = 0\}$ consist of smooth curves in $\Si$ representing $\ww(\nu_\Si)$. 
Moreover, $\{s_\nu = 0\}$ and $\{v_\nu = 0\}$ intersect transversely in $\Si$, 
so that the finite set $\{v_\nu = 0\}\cap \{s_\nu = 0\}$ represents
$\ww(\nu)\cup \ww(\nu) = \ww(\det T\Si)\cup \ww(\nu)$;
\item[(iii)]
$\{s=0\}$ and $\{v_\nu = 0\}$ are disjoint subsets of $\Si$;
\item[(iv)]
$s_\Si$ is a generic section of $T\Si$, so that $\{s_\Si = 0\}$ consists of a finite number
of points representing $\ww(T\Si)$.
\end{enumerate}
  
Given a finite set $X$, denote by $|X|_2\in\Z/2\Z$ the cardinality of $X$ modulo $2$.
Then, we have
\[
\langle\ww(F_v|_\Si), [\Si]\rangle = |\{s=0\}|_2,\quad \langle\ww(T\Si), [\Si]\rangle = |\{s_\Si = 0\}|_2, 
\]
\[ 
\langle\ww(\det T\Si)\cup \ww(\nu_\Si), [\Si]\rangle = |\{v_\nu = 0\}\cap \{s_\nu = 0\}|_2. 
\]
Therefore Equation~\eqref{e:bhWSF} is equivalent to the following equality:
\begin{equation}\label{e:equal}
|\{s=0\}|_2 = |\{s_\Si = 0\}|_2 + |\{v_\nu = 0\}\cap \{s_\nu = 0\}|_2. 
\end{equation}
The finite set $\{s_\Si = 0\}$ can be tautologically decomposed as a disjoint union:
\[
\{s_\Si = 0\} = (\{v_\nu = 0\}\cap \{s_\Si = 0\})\amalg (\{v_\nu \neq 0\}\cap \{s_\Si = 0\}).
\]
We claim that 
\[
\{v_\nu \neq 0\}\cap \{s_\Si = 0\} = \{s=0\}.
\]
In fact, by Assumption~(iii) above we have 
\[
\{s=0\} = \{v_\nu \neq 0\}\cap \{s=0\},
\]
and clearly
\[
\{v_\nu \neq 0\}\cap \{s=0\}\subset  \{v_\nu \neq 0\}\cap \{s_\Si = 0\}.
\]
On the other hand, if $x\in \{v_\nu \neq 0\}\cap \{s_\Si = 0\}$ then $s(x)=0$  
because, since $v_\nu(x)\neq 0$, 
the projection $F_v(x) \to T_x\Si$ is an isomorphism. Thus, the claim is proved. 
In order to establish Equality~\eqref{e:equal} it is now enough to check that
\begin{equation}\label{e:equal2}
|\{v_\nu = 0\}\cap \{s_\Si = 0\}|_2 = |\{v_\nu = 0\}\cap \{s_\nu = 0\}|_2.
\end{equation}
Let $C$ be the collection of smooth curves $\{v_\nu = 0\}\subset\Si$. At each $x\in C$
we have a splitting 
\[
F_v(x) = (F_v(x)\cap T_x\Si)\oplus \nu_\Si(x), 
\]
therefore the restriction $F_v|_C$ splits as a sum of line bundles
\[
F_v|_C = \lambda \oplus  \nu_\Si |_C,
\]
where $\lambda = \{F_v(x)\cap T_x\Si\}_{x\in C}$. We claim that the line bundles $\lambda$ and $\nu_\Si |_C$ are isomorphic. In fact, along each component 
of $C$ the bundle $F_v$ is trivial because it is oriented, so the two line bundles are 
either both trivial or both non-trivial. Thus, 
$\langle\ww(\lambda), [C]\rangle = \langle\ww(\nu_\Si|_C), [C]\rangle$, and 
Equality~\eqref{e:equal2} follows from the observation that the restriction of $s_F$ and $s_\nu$ to $C$ are 
generic sections of, respectively, $\lambda$ and $\nu_\Si |_C$. 
\end{proof} 

\section{Second bare hands proof of Theorem~\ref{t:main}}\label{s:barehands2} 
The aim of this section is to provide a genuine proof of Theorem~\ref{t:main} using minimal background, 
employing some of the ideas we summarized in Section~\ref{ss:FM-proof}.  
Let us first outline an elementary proof of the last portion of the proof presented in~\cite{FM}. 

\begin{lem}\label{l:even-4par} Let  $N=\chi(S^3, L)$ be a $3$-manifold obtained by surgery along a framed link
$L\subset S^3$ such that all framings are {\rm even}. Let $W$ be the corresponding $4$-manifold
obtained by attaching $4$-dimensional $2$-handles to the $4$-ball, so that $N=\partial W$. 
Then, $W$ is parallelizzable.
\end{lem} 

\begin{proof} We refer the reader to~\cite{FM} for further details. For simplicity, assume that $L$ is a one-component
link with even framing $n$. As we can assume that the attaching tubes of the $2$-handles are pairwise
disjoint, this is not really restrictive.  Let $N(L)\subset \partial D^4$ be the attaching tube of the corresponding $2$-handle
attached to $D^4$. Both $D^4$ and $D^2\times D^2$ are parallelizable,
so we have to show that they carry some framings which match on $N(L)$.
Fix a reference framing $\Ff_0$ on $TD^4$. Then, the restriction to $N(L)$ of any framing $\Ff$ on the $2$-handle
is encoded by a map $\rho: N(L)\to SO(4)$.  Viewing $S^3$ as the group of unit quaternions 
one can construct a 2-fold covering map $S^3\x S^2\to SO(4)$ showing that $\pi_1(SO(4))= \Z/2\Z$. 
As the solid torus $N(L)$ retracts onto $L\cong S^1$, $\rho$ determines an element $\overline\rho\in\Z/2Z$ which 
vanishes if and only if the two framings coincide on $N(L)$. It is easy to see that $\overline\rho$ is equal to 
$n\bmod 2$. 
\end{proof}

\begin{cor}\label{c:stab-par} If a $4$-manifold $W$ is parallelizable, then $\partial W$ is {\rm stably-parallelizable}.
In fact, the Whitney sum of the tangent bundle $T\partial W$ with a trivial line bundle $\epsilon$ is a product bundle.
\end{cor} 

\begin{proof}
By the existence of a collar of $\partial W$ in $W$ it is immediate that $T\partial W\oplus\epsilon\cong TW|_{\partial W}$.
\end{proof}

%Finally we have
\begin{lem}\label{l:stably-almost} If a closed, connected, orientable $3$-manifold $N$ 
is stably-parallizable, then it admits a quasi framing, hence it is parallelizable.
\end{lem}

\begin{proof} 
We reproduce the short bare hands  argument  of ~\cite[Lemma~3.4]{KM}. With the usual notation, let
$N_0 = N\setminus {\rm Int}(B)$.  Since $TN_0$ is oriented, a bundle isomorphism 
$TN_0\oplus\ep\cong\ep^4$ gives rise to a map from $N_0$ to the Grassmannian $Gr(3,4)$ of 
oriented $3$-planes in $\R^4$. Since $Gr(3,4)\cong S^3$ and  $N_0$ has a $2$-dimensional spine, by transversality
any such map is not surjective up to homotopy, hence it is homotopically trivial, therefore $TN_0$ is trivial.  
\end{proof}

The following lemma is trivial. 

\begin{lem}\label{l:trivial}
Let $M$ and $M'$ be closed, connected, oriented $3$-manifolds. 
If $M\# M'$ is parallelizable, then both $M$ and $M'$ admit a quasi framing, hence 
they are parallelizable.
\end{lem}

\begin{proof}
Let $N = M\# M'$. Obviously $M_0$ embeds into $N$ and $TM_0$ is the restriction of $TN$ to $M_0$. 
The same holds for $M'$.
\end{proof} 

Combining Corollary~\ref{c:stab-par} with Lemmas~\ref{l:stably-almost} and~\ref{l:trivial}, 
to complete our second bare hands proof of Theorem~\ref{t:main} we are 
reduced to providing a proof using minimal background of the following proposition.

\begin{prop}\label{p:mainprop} For every connected, closed, oriented $3$-manifold $M$,
there exists another such $3$-manifold $M'$ such that  $N=M\# M'$ is of the form $N=\chi(S^3,L)$ for some
framed link $L\subset S^3$ such that all framings are even.
\end{prop}

\begin{proof} 
We use some basic facts about Heegaard splittings of $3$-manifolds.
Let us start with any Heegaard splitting of $M$ of some genus $g$. Up to diffeomorphisms, 
$M_0$ can be relized as follows. Given a handlebody $\HG_g$ of genus $g$, the orientable 
surface $\Sigma_g = \partial \HG_g$ contains a non separating system $C=\{c_1,\dots, c_g\}$ of 
$g$ pairwise disjoint smooth circles. A tubular neighbourhood $N(C)$ in $\Sigma_g$ is formed by a system of 
pairwise disjoint attaching tubes for $3$-dimensional $2$-handles, which, when attached to $\HG_g$  
give $3$-manifold $M_0$. The closed $3$-manifold $M$ is obtained by attaching a further final $3$-handle.
The union of the above $2$- and $3$-handles gives the second handlebody $\HG'_g$ of the Heegaard
splitting, glued to $\HG_g$ along the common boundary $\Sigma_g$. Fix any standard embedding
of $\HG_g$ into $S^3$, so that the closure of $S^3\setminus \HG_g$ is a handlebody as well. 
This embedding realizes a genus-$g$ Heegaard splitting of $S^3$. The collection of curves 
$C\subset \partial \HG_g \subset S^3$ becomes a link $L$ in $S^3$, with each component of $L$ framed 
by a parallel curve in $\partial \HG_g$. Now we can apply the key basic Lemma 1 of \cite{Ro}, 
which has a bare hands proof. In our situation, the lemma implies that $$\chi(S^3,L)=M\cs M'$$
for some $3$-manifold $M'$. It is an immediate consequence of the above description of $M_0$ and of the definition
of surgery along a framed link that $\chi(S^3,L)$ is obtained by gluing $M_0$ and $M'_0$ along
their spherical boundaries. Applying Smale's theorem~\cite{Sm} we can conclude that
$\chi(S^3,L)=M\cs M'$. 

Now fix a complete system $\Mm =\{m_1,\dots, m_g\} $ of meridians of $\HG_g$. The curves $m_i$ 
bound a system of disjoint $2$-disks properly embedded into $(\HG_g,\Sigma_g)$. Denote by $\tau_i$
the Dehn twist on $\Sigma_g$ along $m_i$. Since every $\tau_i$ extends to a diffeomorphism $\overline\tau_i$
of the whole $\HG_g$, we can modify a given embedding of $\HG_g$ into $S^3$ by applying
any finite sequence of such $\overline\tau_i$'s. So we are reduced to show that in this way we can 
obtain an embedding such that the framing of each $c_i$ determined as above by 
the embedding in $\Sigma_g$ is even.  This is the content of~\cite[Lemma~8.4.1]{BP}
(proved  therein to have a treatment with bare hands of Kaplan's result for the double 
$D(M)=-M \cs M$ of $M$). The proof of~\cite[Lemma~8.4.1]{BP} boils down to solving 
a certain $\Z/2\Z$-linear system.
\end{proof}

\begin{remarks} 
(1) Smale's theorem~\cite{Sm} is an essential ingredient of Rourke's clever proof of the Lickorish-Wallace theorem.
We refer to it also in the proof of Proposition~\ref{p:mainprop}. However, this is not really necessary. 
In fact, the description of $\chi(S^3,L)$ as obtained by gluing together $M_0$ and $M'_0$ along their spherical 
boundaries suffices. Thus,  {\em Smale's theorem can be discarded from the background of our second bare hands 
proof of Theorem~\ref{t:main}}.
(2) By \cite{Wa}, the Lickorish-Wallace theorem is bare hands equivalent to $``\Omega_3=0"$.
Modulo Smale's theorem, Rourke's proof is the simplest one. However, a more basic proof that 
$\Omega_3=0$ could probably be concocted by combining a bare hands proof that $M$ is parallelizable
with a specialization of the elementary proof of a theorem by Thom given in~\cite{BH}.
\end{remarks}

\section{Third bare hands proof of Theorem~\ref{t:main}}\label{s:barehands3} 

We shall make use of Lemma~\ref{l:seiferttwist} below, which could be viewed as a `ground zero' 
fact about spin structures. Let $N$ be an oriented $3$-manifold, $K\subset N$ an oriented knot and 
$n:K\to TN|_K$ a unitary normal vector field along $K$.
The orientation of $K$ determines the unitary tangent vector field $t:K\to TN|_K$ and an orthonormal oriented 
framing $\Ff_n = (t,n,b)$ of $TN|_K$. Let $F\subset N$ be a smoothly embedded, oriented surface with 
$\del F=K$. Since $F$ retracts onto a one-dimensional CW-complex, $TF$ is trivial. Let $(a,b)$ be any  
oriented, orthonormal framing of $TF$ and $(a,b,c)$ the orthonormal framing of $TN|_F$ obtained by adding the oriented 
unit normal vector field $c$ to $F$. From now on, we shall implicitely use the framing $(a,b,c)$ to identify, at any point of $F$, 
the set of orthonormal framings of $TN$ with $SO(3)$ and the set of unit vectors of $TN$ with $S^2$. 
Define the map $\varphi_n: K\to S^1$ by $\varphi_n(x) = e^{i\theta(x)}$, where $\theta(x)$ is the 
counterclockwise angle between $c(x)$ and $n(x)$ measured in the oriented normal plane to $K$ at $x$. 

\begin{lem}\label{l:seiferttwist}
The framing $\Ff_n$ of $TN|_K$ extends to a framing of $TN|_F$ if and only if $\deg(\varphi_n)$ is odd. 
\end{lem} 

\begin{proof} 
Let $\psi_n:K\to SO(3)$ be the map given by $\psi_n(x)=\Ff_n(x)$. Clearly $\Ff_n$ extends to a framing of $TN|_F$ 
if and only if $\psi_n$ extends to a map $F\to SO(3)$, which happens if and only if the image of $[K]\in H_1(K;\Z/2\Z)$ 
under the induced map $(\psi_n)_*:H_1(K;\Z/2\Z)\to H_1(SO(3);\Z/2\Z)\cong \Z/2\Z$ is trivial.  
Consider the $S^1$-fibrations $\pi_1,\pi_2:SO(3)\to S^2$ given by 
$\pi_1(a',b',c')= a'$ and $\pi_2(a',b',c')= b'$ and homotope $\Ff_n$ until there are two disjoint intervals 
$A,B\subset K$ such that $n=c$ on $K\setminus A$ 
and $t=a$ on $K\setminus B$. Then, setting $C := \pi_1^{-1}(\sm{1\\0\\0})$ and $C':=\pi_2^{-1}(\sm{0\\1\\0})$, 
it is easy to check that  
\[
(\psi_n)_*([K]) = \chi(F)[C]+\deg(\varphi_n)[C']\in H_1(SO(3);\Z/2\Z).
\]
Since $[C]=[C']$ is a generator of $H_1(SO(3);\Z/2\Z)$, we deduce that $\psi_n$ 
extends to $F$ if and only if $\chi(F) + \deg(\varphi_n)$ is even. But $\chi(F)$ 
is always odd, therefore the statement holds.  
\end{proof} 

Fix a Heegaard splitting $M=\HG_g\cup\HG'_g$ and let $C=\{c_1,\dots, c_g\}\subset \del\HG_g=\del\HG'_g$ 
be a complete system of meridians for $\HG'_g$.  Consider a standard embedding of $\HG_g$ in $\R^3$ and unit vector 
field $n_i$ along the curves $c_i\subset \partial \HG_g$, normal to $\del\HG_g$ and pointing towards $\HG_g$.  
As in the proof of Proposition~\ref{p:mainprop}, using~\cite[Lemma~8.4.1]{BP} we can choose the embedding 
so that each $n_i$ defines an even 
framing of $c_i$ with respect to the Seifert framing in $\R^3$. Note that, by Lemma~\ref{l:seiferttwist}, 
this is equivalent to saying that the induced framing $\Ff_{n_i}$ of $T\R^3|_{c_i}$ does not extend to 
a framing of $T\R^3$ over a Seifert surface. The vector fields $n_i$ coincide with the unit normal vector fields 
determined by collars of each curve $c_i$ in the corresponding $2$-disk $D_i$ properly embedded into $(\HG_g', \partial \HG'_g)$.
Let if $B_i\subset M$ be a $3$-disk containing $D_i$. By Lemma~\ref{l:seiferttwist} the framings $\Ff_{n_i}$,  
regarded as framings of $TB_i|_{c_i}$, do not extend to framings of $TB_i|_{D_i}$. On the other hand, 
the restriction of the standard framing $\Ff$ of $\R^3$ to each $c_i$ is homotopic to a framing $\Ff_{m_i}$ determined 
by a unit vector field $m_i$ normal to $c_i$ and defining an odd framing with respect to the Seifert framing. 
Again by Lemma~\ref{l:seiferttwist}, this means that $\Ff$ can be extended along each $D_i$, yielding a quasi-framing of $M$. 
This concludes the third bare hands proof of Theorem~\ref{t:main}.

\end{document}